\documentclass[12pt]{amsproc}
\usepackage{amssymb}
\usepackage{amsmath}
\usepackage{amsfonts}

\theoremstyle{plain}

\newtheorem{lemma}{Lemma}

\newtheorem{theorem}{Theorem}
\numberwithin{equation}{section}
\input{tcilatex}

\begin{document}
\title[Nonlinear elliptic fourth order equations...]{Nonlinear elliptic
fourth order equations existence and multiplicity results}
\author{Mohammed Benalili}
\address{Faculty of Sciences, Dept. Math. University Aboubakr BelKa\"{\i}d.
Tlemcen Algeria.}
\email{m\_benalili@mail.univ-tlemcen.dz}
\author{Kamel Tahri}
\curraddr{Faculty of Sciences, Dept. Math. University Aboubakr BelKa\"{\i}d.
Tlemcen Algeria.}

\begin{abstract}
This paper deals with the existence of solutions to a class of fourth order
nonlinear elliptic equations. The technique used relies on critical points
theory. The solutions appeared as critical points of a functional restricted
to a suitable manifold.In the case of constant coefficients we obtain the
existence of tree distinct solutions.
\end{abstract}

\maketitle

\section{Introduction}

Let$\ (M,g)$ be a Riemannian compact smooth $n-$manifold, $n\geq 5$, with
metric $g$ and scalar curvature $S_{g}$, we let $H_{2}^{2}(M)$ be the
standard Sobolev space which is the completion of the space 
\begin{equation*}
C_{2}^{2}(M)=\left\{ u\in C^{\infty }(M)\text{: }\left\Vert u\right\Vert
_{2,2}<+\infty \right\}
\end{equation*}%
with respect to the norm $\left\Vert u\right\Vert
_{2,2}=\sum_{l=0}^{2}\left\Vert \nabla ^{l}u\right\Vert _{2}$.

In this paper, we investigate solutions of a class of fourth order elliptic
equations, on compact $n$-dimensional Riemannian manifolds, of the form 
\begin{equation}
\Delta ^{2}u+\nabla ^{i}(a(x)\nabla _{i}u)+b(x)u=f(x)\left\vert u\right\vert
^{N-2}u+\lambda \left\vert u\right\vert ^{q-2}u  \tag{1}  \label{1}
\end{equation}%
where $a$, $b$, and $f$ \ are smooth functions on $M$ , $N=\frac{2n}{n-4}$
is the critical exponent, $1<q<2$ \ a real number, $\lambda >0$ a real
parameter.

Consideration for such problem comes from conformal geometry: indeed, in
1983, Paneitz \cite{8} introduced a conformal fourth order operator defined
on $4$-dimensional Riemannian manifolds which was generalized by Branson 
\cite{3} to higher dimensions. 
\begin{equation*}
PB_{g}(u)=\Delta ^{2}u+div(-\frac{(n-2)^{2}+4}{2(n-1)(n-2)}S_{g}.g+\frac{4}{%
n-2}Ric)du+\frac{n-4}{2}Q^{n}u
\end{equation*}%
where $\Delta u=-div\left( \nabla u\right) $, $R$ is the scalar curvature, $%
Ric$ is the Ricci curvature of $g$ and where%
\begin{equation*}
Q^{n}=\frac{1}{2(n-1)}\Delta S_{g}+\frac{n^{3}-4n^{2}+16n-16}{%
8(n-1)^{2}(n-2)^{2}}S_{g}^{2}-\frac{2}{(n-2)^{2}}\left\vert Ric\right\vert
^{2}
\end{equation*}%
is associated to the notion of $Q$ -curvature.

We refer to a Paneitz-Branson type operator as an operator of the form%
\begin{equation*}
P_{g}u=\Delta ^{2}u+\nabla ^{i}(a(x)\nabla _{i}u)+h(x)u\text{.}
\end{equation*}

Equation (\ref{1}) is a perturbation of the equation 
\begin{equation}
\Delta ^{2}u+\nabla ^{i}(a(x)\nabla _{i}u)+h(x)u=f(x)\left\vert u\right\vert
^{N-2}u\text{.}  \tag{2}  \label{2}
\end{equation}%
Since the embedding $H_{2}^{2}\hookrightarrow H_{N}^{k}$ , ($\ k=0$, $1$)
fails to be compact, as known, one encounters serious difficulties in
solving equations like (\ref{1}).

Since 1990 many results have been established for precise functions $a$, $h$
and $f.$ D.E. Edmunds, D. Fortunato, E. Jannelli$\left( \text{\cite{8}}%
\right) $ proved for $n\geq 8$ that if $\lambda \in (0,\lambda _{1})$, with $%
\lambda _{1}$ is the first eigenvalue of $\Delta ^{2}$ on the euclidean open
ball $B$, the problem 
\begin{equation*}
\left\{ 
\begin{array}{c}
\Delta ^{2}u-\lambda u=u\left\vert u\right\vert ^{\frac{8}{n-4}}\text{ in }B
\\ 
u=\frac{\partial u}{\partial n}=0\text{ on }\partial B%
\end{array}%
\right.
\end{equation*}%
has a non trivial solution.

In 1995, R$.$ Van der Vorst (\cite{8}) obtained the same results as D.E.
Edmunds, D. Fortunato, E. Jannelli. when applied to the problem%
\begin{equation*}
\left\{ 
\begin{array}{c}
\Delta ^{2}u-\lambda u=u\left\vert u\right\vert ^{\frac{8}{n-4}}\text{ in }%
\Omega \\ 
u=\Delta u=0\text{ \ on \ }\partial \Omega%
\end{array}%
\right.
\end{equation*}%
where $\Omega $ is an open bounded set of $R^{n}$ and moreover he showed
that the solution is positive

In $\left( \text{\cite{7}}\right) $ D.Caraffa studied the equation%
\begin{equation}
\Delta ^{2}u+\nabla ^{i}(a(x)\nabla _{i}u)+b(x)u=f(x)\left\vert u\right\vert
^{N-2}u  \tag{3}  \label{3}
\end{equation}%
in the case $f(x)=$constant; and in the particular case where the functions $%
a(x)$ and $b(x)$ are precise constants she obtained the existence of
positive regular solutions.

Let $f$ \ be a $C^{\infty }$ function on $\ M$, $f^{-}=-\inf (f,0)$, $%
f^{+}=\sup (f,0)$ and for $a$,$b$, $\ f$ , $C^{\infty }$ -functions on $M$,
we let $\ \ $%
\begin{equation*}
\lambda _{a,f}=\underset{u\in A}{\inf }\frac{\int_{M}(\Delta
u)^{2}dv_{g}-\int_{M}a\left\vert \nabla u\right\vert ^{2}dv_{g}}{%
\int_{M}u^{2}dvg}
\end{equation*}%
where $A=\left\{ u\in H_{2}\text{, }u\geq 0\text{, }u\not\equiv 0\text{ s.
t. }\int_{M}f^{-}udv_{g}=0\right\} ,~\ $and%
\begin{equation*}
\lambda _{a,f}=+\infty \ \ \ \ \text{if\ \ \ \ \ \ \ \ \ }A=\phi \text{ }
\end{equation*}%
the second author establish the following:

\begin{theorem}
\cite{3} Let $a$, $h$ be $C^{\infty }$ functions on $M$ with $h$ negative.
For every $C^{\infty }$ function, $f$ on $M$ with $\int_{M}f^{-}dv_{g}>0$,
there exists a constant $C>0$ which depends only on $\frac{f^{-}}{\int
f^{-}dv_{g}}$ such that if $f$ satisfies the following conditions 
\begin{equation*}
\left. 
\begin{array}{l}
(1)\text{ }\left\vert h(x)\right\vert <\lambda _{a,f}\text{ \ \ \ \ \ \ \ \
\ \ \ \ \ \ \ \ \ for any }x\in M\text{\ \ \ \ \ \ \ \ \ \ \ \ \ \ \ \ \ \ \
\ \ \ \ \ \ \ \ \ \ \ \ \ \ \ \ \ \ \ \ \ \ \ \ \ \ \ \ \ \ \ \ \ \ \ \ \ \
\ \ \ \ \ \ \ \ \ \ \ \ \ \ \ \ } \\ 
(2)\text{ }\frac{\sup f^{+}}{\int f^{-}dv_{g}}<C \\ 
(3)\text{ }\sup_{M}f>0,%
\end{array}%
\right.
\end{equation*}%
then the subcritical equation 
\begin{equation*}
\Delta ^{2}u+\nabla ^{i}(a\nabla _{i}u)+hu=f\left\vert u\right\vert ^{q-2}u,%
\text{ \ \ \ \ }q\in \left] 2,N\right[
\end{equation*}%
has at least two distinct solutions $u$ and $v$ satisfying $F_{q}\left(
u\right) <0<F_{q}\left( v\right) $ and of class $C^{4,\alpha }$, for some $%
\alpha \in (0,1)$.
\end{theorem}

\begin{theorem}
\cite{3} Let $a$, $h$ be $C^{\infty }$ functions on $M$ with $h$ negative.\
For every $C^{\infty }$ function $f$ on $M$ with $\int_{M}f^{-}dv_{g}>0$
there exists a constant $C>0$ which depends only on $\frac{f^{-}}{\int
f^{-}dv_{g}}$ such that if $f$ satisfies the following conditions 
\begin{equation*}
\left. 
\begin{array}{l}
(1)\text{ }\left\vert h(x)\right\vert <\lambda _{a,f}\text{ \ \ \ \ \ \ \
for any }x\in M\text{\ \ \ \ \ \ \ \ \ \ \ \ \ \ \ \ \ \ \ \ \ \ \ \ \ \ \ \
\ \ \ \ \ \ \ \ \ \ \ \ \ \ \ \ \ \ \ \ \ \ \ \ \ \ \ \ \ \ \ \ \ \ \ \ \ \
\ \ \ \ \ \ \ \ \ \ \ \ \ \ \ \ \ \ \ \ \ \ \ \ \ } \\ 
(2)\text{ }\frac{\sup f^{+}}{\int f^{-}dv_{g}}<C%
\end{array}%
\right.
\end{equation*}%
the critical equation%
\begin{equation*}
\Delta ^{2}u+\nabla ^{i}(a\nabla _{i}u)+hu=f\left\vert u\right\vert ^{N-2}u
\end{equation*}%
has a solution of class $C^{4,\theta }$, for some $\theta \in (0,1)$, \ with
negative energy.
\end{theorem}

Let 
\begin{equation}
\Delta ^{2}u+\nabla ^{i}(a(x)\nabla _{i}u)+h(x)u=f(x)\left\vert u\right\vert
^{N-2}u+\lambda \left\vert u\right\vert ^{q-2}u+\epsilon g(x)  \tag{4}
\label{4}
\end{equation}%
where $a$, $h$, $f$ and $g$ are smooth functions on $M$ , $N=\frac{2n}{n-4}$
is the critical exponent, $2<q<N$ \ a real number, $\lambda >0$ a real
parameter and $\epsilon >0$ any small real number the second author obtain

\begin{theorem}
\cite{4} Let $(M,g)$ be a compact Riemannian $n-$manifold, $n\geq 6$, $a$, $%
h $, $f$, $g$ be smooth real \ functions on $M$ with

(i) $f(x)>0$ and $g(x)>0$ everywhere on $M$

(ii) the operator $P(u)=\Delta ^{2}u+\nabla ^{i}(a(x)\nabla _{i}u)+h(x)u$ is
coercive

(iii) if $n>6$ , we suppose $\frac{\Delta f(x_{o})}{2f(x_{o})}%
+C_{1}(n)S_{g}(x_{o})+C_{2}(n)a(x_{o})>0$ and if $n=6$, we suppose that $%
\frac{4}{3n}S_{g}(x_{o})+\frac{1}{(n-4)}a(x_{o})>0$, where $C_{1}(n)=\frac{%
5n^{2}(n-7)+52(n-1)}{6n(n+2)(5n-6)}$ and $\ C_{2}(n)=\frac{8(n-1)}{(n+2)(n-6)%
}$

Then equation (\ref{4}) has at least two distinct solutions in $H_{2}$
\end{theorem}

Our main results in this paper state as follows

\begin{theorem}
\label{th1} Let $\left( M,g\right) $ be an $n$-dimensional compact
Riemannian manifold with $n\geqslant 6$ and $f$ a positive smooth function
on $M$. Assume that the operator $P(u)=\Delta ^{2}u+\nabla ^{i}(a(x)\nabla
_{i}u)+b(x)u$ is coercive.

If $n>6$ and at the point $x_{o}$ where the function $f$ achieves its
maximum the following condition is satisfied 
\begin{equation*}
\frac{n\left( n^{2}+4n-20\right) }{2\left( n-2\right) \left( n^{2}-4\right) }%
S_{g}\left( x_{o}\right) +\frac{n\left( n-6\right) }{\left( n-2\right)
\left( n^{2}-4\right) }a\left( x_{o}\right) -\frac{n}{8\left( n-2\right) }%
\frac{\Delta f(x_{o})}{f\left( x_{o}\right) }>0
\end{equation*}%
then there exists a $\Lambda >0$ such that for any $\lambda \in \left(
0,\Lambda \right) $, equation (\ref{1}) have a non trivial solution of class 
$C^{4,\theta }(M)$, $\alpha \in (0,1)$.

If $n=6$ and the condition $S_{g}(x_{o})>-3a(x_{o})$ is satisfied, then
equation (\ref{1}) has a solution $u$ of class $C^{4,\theta }(M)$ provided
that $\lambda \in \left( 0,\Lambda \right) $.
\end{theorem}

In case of constant coefficients, equation (\ref{1}) reduces to 
\begin{equation}
\Delta ^{2}u+\alpha \Delta u+\beta u=f(x)\left\vert u\right\vert
^{N-2}u+\lambda \left\vert u\right\vert ^{q-2}u  \tag{5}  \label{5}
\end{equation}%
where $\alpha $ and $\beta $ are real constants; we get the existence of
tree solutions.

\begin{theorem}
\label{th2} Suppose all the conditions of Theorem\ref{th1} are satisfied and
moreover $\frac{\alpha ^{2}}{4}>\beta $ with $\alpha >0$. Then equation (\ref%
{5}) has beside a positive $u^{+}$ and a negative $u^{-}$ smooth solutions a
third solution $w$ distinct from $u^{+}$ and $u^{-}$.
\end{theorem}

The technique relies on critical points theory. To find solutions we use a
method developed in \cite{1}. The solutions appeared as critical points
restricted to a suitable manifold. In the case of constant coefficients we
obtain the existence of two s

Consider the functionals $J_{\lambda }$, $J_{\lambda }^{+}$ and $J_{\lambda
}^{-}$ defined on $H_{2}$ by 
\begin{equation}
J_{\lambda }(u)=\frac{1}{2}\left( \left\Vert \Delta u\right\Vert
_{2}^{2}-\int_{M}a(x)\left\vert \nabla u\right\vert
^{2}dv_{g}+\int_{M}b(x)u^{2}dv_{g}\right) -\frac{\lambda }{q}%
\int_{M}\left\vert u\right\vert ^{q}dv_{g}\text{\ \ \ }  \tag{6}  \label{6}
\end{equation}%
\begin{equation*}
-\frac{1}{N}\int_{M}f(x)\left\vert u\right\vert ^{N}dv_{g},
\end{equation*}
\begin{equation}
J_{\lambda }^{+}(u)=\frac{1}{2}\left( \left\Vert \Delta u\right\Vert
_{2}^{2}-\int_{M}a(x)\left\vert \nabla u\right\vert
^{2}dv_{g}+\int_{M}b(x)u^{2}dv_{g}\right) -\frac{\lambda }{q}\int_{M}\left(
u^{+}\right) ^{q}dv_{g}\text{\ \ \ }  \tag{6}
\end{equation}%
\begin{equation*}
-\frac{1}{N}\int_{M}f(x)\left( u^{+}\right) ^{N}dv_{g}
\end{equation*}%
and 
\begin{equation}
J_{\lambda }^{-}(u)=\frac{1}{2}\left( \left\Vert \Delta u\right\Vert
_{2}^{2}-\int_{M}a(x)\left\vert \nabla u\right\vert
^{2}dv_{g}+\int_{M}b(x)u^{2}dv_{g}\right) -\frac{\lambda }{q}%
\int_{M}\left\vert u^{-}\right\vert ^{q}dv_{g}\text{\ \ \ }  \tag{6}
\end{equation}%
\begin{equation*}
-\frac{1}{N}\int_{M}f(x)\left\vert u^{-}\right\vert ^{N}dv_{g}
\end{equation*}%
where $u^{+}=\max \left( u,0\right) $ and $u^{-}=\min \left( u,0\right) $.

Let

\begin{equation*}
Q_{\lambda }\left( u\right) =\left\langle \nabla J_{\lambda }\left( u\right)
,u\right\rangle
\end{equation*}

and

\begin{equation*}
Q_{\lambda }^{\pm }\left( u\right) =\left\langle \nabla J_{\lambda }^{\pm
}\left( u\right) ,u\right\rangle
\end{equation*}%
where $\left\langle \nabla J_{\lambda }\left( u\right) ,v\right\rangle $
denotes the value of $\nabla J_{\lambda }\left( u\right) $ at $v$.

We consider also the set%
\begin{equation*}
M_{\lambda }=\left\{ u\in H_{2}:Q_{\lambda }\left( u\right) =0\text{ and }%
\left\Vert u\right\Vert \geq \rho >0\right\}
\end{equation*}%
and%
\begin{equation*}
M_{\lambda }^{\pm }=\left\{ u\in H_{2}:Q_{\lambda }^{\pm }\left( u\right) =0%
\text{ and }\left\Vert u\right\Vert \geq \rho >0\right\} \text{.}
\end{equation*}
Along this paper the functions $a$ and $b$ are taken such that

\begin{equation*}
\left\Vert u\right\Vert ^{2}=\left\Vert \Delta u\right\Vert
_{2}^{2}-\int_{M}a(x)\left\vert \nabla u\right\vert
^{2}dv_{g}+\int_{M}b(x)u^{2}dv_{g}
\end{equation*}%
is a norm on $H_{2}^{2}(M)$ equivalent to the usual one: for example by
letting $\max_{x\in M}a(x)<0$ and $\min_{x\in M}b(x)>0$ which is equivalent
to assume that the operator $P(u)=\Delta ^{2}u+a(x)\Delta u+b(x)u$ is
coercive.

First we ensure that the manifold $M_{\lambda }$ ( respectively $M_{\lambda
}^{\pm }$ ) is not empty.

\begin{lemma}
\label{lem1} There is a real $\lambda _{o}>0$ such that the set $M_{\lambda
} $ is non empty for any $\lambda \in \left( 0,\lambda _{o}\right) $.
\end{lemma}

\begin{proof}
For $u\in H_{2}^{2}\left( M\right) $ with $\left\Vert u\right\Vert
_{2,2}\geq \rho >0$ and $t>0$, 
\begin{equation*}
Q_{\lambda }\left( tu\right) =t^{2}\left\Vert u\right\Vert ^{2}-\lambda
t^{q}\left\Vert u\right\Vert _{q}^{q}-t^{N}\int_{M}f(x)\left\vert
u\right\vert ^{N}dv_{g}\text{.}
\end{equation*}%
Put 
\begin{equation*}
\alpha \left( t\right) =\left\Vert u\right\Vert
^{2}-t^{N-2}\int_{M}f(x)\left\vert u\right\vert ^{N}dv_{g}
\end{equation*}%
and 
\begin{equation*}
\beta \left( t\right) =\lambda t^{q-2}\left\Vert u\right\Vert _{q}^{q}\text{.%
}
\end{equation*}%
The Sobolev inequality leads to%
\begin{equation*}
\alpha \left( t\right) \geq \left\Vert u\right\Vert ^{2}-\max_{x\in
M}f\left( x\right) \left( \max \left( \left( 1+\epsilon \right)
K_{o},A_{\epsilon }\right) \right) ^{N}\left\Vert u\right\Vert
_{H_{2}^{2}\left( M\right) }^{N}t^{N-2}
\end{equation*}%
where 
\begin{equation*}
\frac{1}{K_{\circ }}=\inf {}_{u\in H_{2}^{2}(%
%TCIMACRO{\U{211d} }%
%BeginExpansion
\mathbb{R}
%EndExpansion
^{n})-\{0\}}\frac{\left\Vert \Delta u\right\Vert _{2}^{2}}{\left\Vert
u\right\Vert _{N}^{2}}
\end{equation*}
is the best constant in the Sobolev's embedding $H_{2}^{2}(R^{n})\subset
L^{N}(R^{n})$ (see T. Aubin \cite{2} ) and $A_{\epsilon }$ is a positive
constant depending on $\epsilon $, and since the norms $\left\Vert
.\right\Vert $ and $\left\Vert .\right\Vert _{H_{2}^{2}\left( M\right) }$
are equivalent, there is $\Lambda >0$ such that 
\begin{equation}
\alpha \left( t\right) \geq \left\Vert u\right\Vert ^{2}-\Lambda ^{-\frac{N}{%
2}}\max_{x\in M}f\left( x\right) \left( \max \left( \left( 1+\epsilon
\right) K_{o},A_{\epsilon }\right) \right) ^{N}\left\Vert u\right\Vert
^{N}t^{N-2}\text{.}  \tag{7}  \label{7}
\end{equation}%
Combining the H\"{o}lder and the Sobolev inequalities and taking account of
the equivalence of the norms $\left\Vert .\right\Vert $ and $\left\Vert
.\right\Vert _{H_{2}^{2}\left( M\right) }$, we get%
\begin{equation}
\beta \left( t\right) \leq \lambda V\left( M\right) ^{1-\frac{2}{N}}\left(
\max \left( \left( 1+\epsilon \right) K_{o},A\right) \right) ^{\frac{n}{2}%
}\left\Vert u\right\Vert ^{q}t^{q-2}\text{.}  \tag{8}  \label{8}
\end{equation}%
Let $\alpha _{1}(t)$ and $\beta _{1}(t)$ denote respectively right hand
sides of inequalities (\ref{7}) and (\ref{8}); $\alpha _{1}\left( t\right) $
vanishes for%
\begin{equation}
t_{o}=\frac{1}{\left\Vert u\right\Vert \left( \max_{x\in M}f(x)\right) ^{%
\frac{1}{N-2}}\left( \max \left( \left( 1+\epsilon \right) K_{o},A_{\epsilon
}\right) \right) ^{\frac{N}{2\left( N-2\right) }}}\text{.}  \tag{9}
\label{9}
\end{equation}%
So if we choose 
\begin{equation*}
\left\Vert u\right\Vert =\frac{1}{\left( \max_{x\in M}f(x)\right) ^{\frac{1}{%
N-2}}\left( \max \left( \left( 1+\epsilon \right) K_{o},A_{\epsilon }\right)
\right) ^{\frac{N}{2\left( N-2\right) }}}
\end{equation*}%
$t_{o}=1$.

Now since $\alpha _{1}(t)$ and $\beta _{1}\left( t\right) $ are both
decreasing functions, we get 
\begin{equation*}
\min_{x\in \left( 0,\frac{1}{2}\right) }\alpha _{1}\left( t\right) =\alpha
_{1}(\frac{1}{2})=\frac{\left( 1-2^{2-N}\right) }{\left( \max_{x\in
M}f(x)\right) ^{\frac{2}{N-2}}\left( \max \left( \left( 1+\epsilon \right)
K_{o},A_{\epsilon }\right) \right) ^{\frac{N}{\left( N-2\right) }}}
\end{equation*}%
\begin{equation*}
\geq \left( 1-2^{2-N}\right) \rho ^{2}>0
\end{equation*}%
and 
\begin{equation*}
\min_{x\in \left( 0,\frac{t_{o}}{2}\right) }\beta _{1}\left( t\right) =\beta
_{1}(\frac{1}{2})=\frac{2^{2-q}\lambda V(M)^{\left( 1-\frac{2}{N}\right) }}{%
\left( \max_{x\in M}f(x)\right) ^{\frac{q}{N-2}}\left( \max \left( \left(
1+\epsilon \right) K_{o},A_{\epsilon }\right) \right) ^{\frac{q-2+N}{N-2}}}>0%
\text{.}
\end{equation*}%
The equation $Q_{\lambda }(u)=0$ admits a solution if $\min_{x\in \left( 0,%
\frac{1}{2}\right) }\alpha _{1}\left( t\right) \geq \min_{x\in \left( 0,%
\frac{1}{2}\right) }\beta _{1}\left( t\right) $ that is to say if 
\begin{equation*}
0<\lambda <\lambda _{o}=\frac{\left( 2^{q-2}-2^{q-N}\right) V(M)^{\left( 1-%
\frac{2}{N}\right) }}{\left( \max_{x\in M}f(x)\right) ^{\frac{q-2}{N-2}%
}\left( \max \left( \left( 1+\epsilon \right) K_{o},A_{\epsilon }\right)
\right) ^{\frac{q-2}{N-2}}}\text{.}
\end{equation*}%
Hence the set $M_{\lambda }$ is nonempty for any $\lambda \in \left(
0,\lambda _{o}\right) $.

Since $Q_{\lambda }^{\pm }(u)\geqslant Q_{\lambda }(u)$ the same
calculations lead to the conclusion.
\end{proof}

\section{Palais-Smale conditions}

\begin{lemma}
\label{lem2} There is $A>0$ such that $J_{\lambda }(u)\geq A>0$ ( resp. $%
J_{\lambda }(u)\geq A>0$ ) for any $u\in M_{\lambda }$ with $\lambda \in
\left( 0,\min \left( \lambda _{o},\lambda _{1}\right) \right) $ where $%
\lambda _{1}=\frac{\frac{(N-2)\Lambda ^{-\frac{q}{2}}}{2\left( N-q\right) }}{%
V(M)^{1-\frac{2}{N}}\left( \max \left( \left( 1+\epsilon \right)
K_{o},A_{\epsilon }\right) \right) ^{\frac{q}{2}}\rho ^{q-2}}$and $\lambda
_{o}$ is as in Lemma \ref{1}.
\end{lemma}

\begin{proof}
Let $u\in M_{\lambda }$, then 
\begin{equation*}
\left\Vert u\right\Vert ^{2}=\lambda \int_{M}\left\vert u\right\vert
^{q}dv_{g}+\int_{M}f(x)\left\vert u\right\vert ^{N}dv_{g}
\end{equation*}%
so on $M_{\lambda }$ the functional $J_{\lambda }$ writes

\begin{equation*}
J_{\lambda }\left( u\right) =\frac{N-2}{2N}\left\Vert u\right\Vert
^{2}-\lambda \frac{N-q}{Nq}\int_{M}\left\vert u\right\vert ^{q}dv_{g}\text{.}
\end{equation*}%
Combining H\"{o}lder and Sobolev inequalities we obtain%
\begin{equation*}
J_{\lambda }\left( u\right) \geq \frac{N-2}{2N}\left\Vert u\right\Vert
^{2}-\lambda \frac{N-q}{Nq}V(M)^{1-\frac{2}{N}}\left( \max \left( \left(
1+\epsilon \right) K_{o},A_{\epsilon }\right) \right) ^{\frac{q}{2}%
}\left\Vert u\right\Vert _{H_{2}^{2}\left( M\right) }^{q}
\end{equation*}%
and still taking account of the equivalence of the norms $\left\Vert
.\right\Vert _{H_{2}^{2}\left( M\right) }$ and $\left\Vert .\right\Vert $,
we obtain%
\begin{equation*}
J_{\lambda }\left( u\right) \geq \left( \frac{N-2}{2N}-\lambda \frac{N-q}{Nq}%
\Lambda ^{-\frac{q}{2}}V(M)^{1-\frac{2}{N}}\left( \max \left( \left(
1+\epsilon \right) K_{o},A_{\epsilon }\right) \right) ^{\frac{q}{2}}\rho
^{q-2}\right) \left\Vert u\right\Vert ^{2}
\end{equation*}%
where $\Lambda >0$ is a constant.

Hence if 
\begin{equation*}
0<\lambda <\frac{\frac{N-2}{2\left( N-q\right) }\Lambda ^{-\frac{q}{2}}}{%
V(M)^{1-\frac{2}{N}}\left( \max \left( \left( 1+\epsilon \right)
K_{o},A_{\epsilon }\right) \right) ^{\frac{q}{2}}\rho ^{q-2}}
\end{equation*}%
then%
\begin{equation*}
J_{\lambda }\left( u\right) >0\text{ }
\end{equation*}%
for any $u\in M_{\lambda }$.

Since $J_{\lambda }^{\pm }\left( u\right) \geqslant J_{\lambda }\left(
u\right) $ we have the conclusion to $J_{\lambda }^{\pm }\left( u\right) $.
\end{proof}

\begin{lemma}
\label{lem3} Let $\left( M,g\right) $ be a $n$- Riemannian manifold, $n\geq
5 $. \ The following assertions are true

(i) $\left( \nabla Q_{\lambda }(u),u\right) \leq -A<0$ ( resp. $\left(
\nabla Q_{\lambda }^{\pm }(u),u\right) \leq -A<0$ ) , for $u\in M_{\lambda }$
and any $\lambda \in \left( 0,\min \left( \lambda _{o},\lambda _{1}\right)
\right) $

(ii) The critical points of $J_{\lambda }$ ( resp. $J_{\lambda }^{\pm }$ )
are points of $M_{\lambda }$ ( resp. $M_{\lambda }^{\pm }$ ).
\end{lemma}

\begin{proof}
(i) Let $u\in M_{\lambda }$, then%
\begin{equation*}
\left\Vert u\right\Vert ^{2}=\lambda \int_{M}\left\vert u\right\vert
^{q}dv_{g}+\int_{M}f(x)\left\vert u\right\vert ^{N}dv_{g}
\end{equation*}%
and 
\begin{equation*}
\left\langle \nabla Q_{\lambda }(u),u\right\rangle =2\left\Vert u\right\Vert
^{2}-\lambda q\int_{M}\left\vert u\right\vert
^{q}dv_{g}-N\int_{M}f(x)\left\vert u\right\vert ^{N}dv_{g}
\end{equation*}%
\begin{equation*}
=2\left\Vert u\right\Vert ^{2}-\lambda q\int_{M}\left\vert u\right\vert
^{q}dv_{g}-N\left( \left\Vert u\right\Vert ^{2}-\lambda \int_{M}\left\vert
u\right\vert ^{q}dv_{g}\right)
\end{equation*}%
\begin{equation*}
=\left( 2-N\right) \left\Vert u\right\Vert ^{2}+\lambda \left( N-q\right)
\left\Vert u\right\Vert _{q}^{q}\text{.}
\end{equation*}

The combination of H\"{o}lder and Sobolev inequalities allows us to write%
\begin{equation*}
\left\langle \nabla Q_{\lambda }(u),u\right\rangle \leq \left( 2-N\right)
\left\Vert u\right\Vert ^{2}+\lambda \left( N-q\right) V(M)^{1-\frac{2}{N}%
}\left( \max \left( 1+\epsilon \right) K_{o},A_{\epsilon }\right) ^{\frac{q}{%
2}}\left\Vert u\right\Vert _{H_{2}^{2}\left( M\right) }^{q}
\end{equation*}%
and since the norms $\left\Vert .\right\Vert $ and $\left\Vert .\right\Vert
_{H_{2}^{2}\left( M\right) }$ are equivalent, we get%
\begin{equation*}
\left\langle \nabla Q_{\lambda }(u),u\right\rangle \leq \left( \left(
2-N\right) +\lambda \left( N-q\right) V(M)^{1-\frac{2}{N}}\Lambda
^{q-2}\left( \max \left( 1+\epsilon \right) K_{o},A_{\epsilon }\right) ^{%
\frac{q}{2}}\rho ^{q-2}\right) \left\Vert u\right\Vert ^{2}\text{.}
\end{equation*}%
Hence if 
\begin{equation*}
0<\lambda <\lambda _{1}=\frac{\frac{N-2}{2\left( N-q\right) }\Lambda ^{-%
\frac{q}{2}}}{V(M)^{1-\frac{2}{N}}\left( \max \left( 1+\epsilon \right)
K_{o},A_{\epsilon }\right) ^{\frac{q}{2}}\rho ^{q-2}}
\end{equation*}%
then for any $u\in M_{\lambda }$. 
\begin{equation*}
\left\langle \nabla Q_{\lambda }(u),u\right\rangle <0
\end{equation*}

(ii) \ By the Lagrange multiplicators theorem we get the existence of a real
number $\mu $ such that for any $u\in M_{\lambda }$ 
\begin{equation*}
\nabla J_{\lambda }(u)=\mu \nabla Q_{\lambda }(u)
\end{equation*}%
and by testing at the point $u\in M_{\lambda }$, we obtain%
\begin{equation*}
Q_{\lambda }(u)=\left\langle \nabla J_{\lambda }(u),u\right\rangle =\mu
\left\langle \nabla Q_{\lambda }(u),u\right\rangle
\end{equation*}%
and since $\left\langle \nabla Q_{\lambda }(u),u\right\rangle <0$, we get
necessarily that $\mu =0$;

Hence for any $u\in M_{\lambda }$%
\begin{equation*}
\nabla J_{\lambda }(u)=0\text{.}
\end{equation*}%
The same computations are carried to conclude for $\left\langle \nabla
Q_{\lambda }^{\pm }(u),u\right\rangle $ and $J_{\lambda }^{\pm }$.
\end{proof}

\begin{lemma}
\label{lem4} Let ($u_{n}$)$_{n}$ be a sequence in $M_{\lambda }$ ( resp. $%
M_{\lambda }^{\pm }$ ) such that

\begin{equation*}
J_{\lambda }(u_{n})\leq c\text{ ( resp. }J_{\lambda }^{\pm }(u_{n})\leq c%
\text{ )}
\end{equation*}%
and 
\begin{equation*}
\nabla J_{\lambda }(u_{n})-\mu _{n}\nabla Q_{n}(u_{n})\rightarrow 0\text{ (
resp. }\nabla J_{\lambda }^{\pm }(u_{n})-\mu _{n}\nabla Q_{n}^{\pm
}(u_{n})\rightarrow 0\text{ ).}
\end{equation*}

Suppose that 
\begin{equation*}
c<\frac{2}{nK_{o}^{\frac{n}{4}}\max_{x\in M}f(x)^{\frac{n}{4}-1}}
\end{equation*}%
then there is a subsequence of ($u_{n}$)$_{n}$ converging strongly in $%
H_{2}^{2}(M)$.
\end{lemma}

\begin{proof}
Let $\left( u_{n}\right) _{n}\subset M_{\lambda }$%
\begin{equation*}
J_{\lambda }(u_{n})=\frac{N-2}{2N}\left\Vert u_{n}\right\Vert ^{2}-\lambda 
\frac{N-q}{Nq}\int_{M}\left\vert u_{n}\right\vert ^{q}dv(g)
\end{equation*}
We have \ 
\begin{equation*}
J_{\lambda }(u_{n})\geq \frac{N-2}{2N}\left\Vert u_{n}\right\Vert
^{2}-\lambda \frac{N-q}{Nq}\Lambda ^{-\frac{q}{2}}V(M)^{1-\frac{q}{N}}(\max
((1+\varepsilon )K_{\circ },A_{\varepsilon }))^{\frac{q}{2}}\left\Vert
u_{n}\right\Vert ^{q}
\end{equation*}%
\begin{equation*}
J_{\lambda }(u_{n})\geq \left\Vert u_{n}\right\Vert ^{2}(\frac{N-2}{2N}%
-\lambda \frac{N-q}{Nq}\Lambda ^{-\frac{q}{2}}V(M)^{1-\frac{q}{N}}(\max
((1+\varepsilon )K_{\circ },A_{\varepsilon }))^{\frac{q}{2}}\left\Vert
u_{n}\right\Vert ^{q-2})
\end{equation*}%
with $0<\lambda <\frac{\frac{\left( N-2\right) q}{2\left( N-q\right) }%
\Lambda ^{-\frac{q}{2}}}{V(M)^{1-\frac{q}{N}}(\max ((1+\varepsilon )K_{\circ
},A_{\varepsilon }))^{\frac{q}{2}}\left\Vert u\right\Vert ^{q-2}}$ .

On the other hand, we have 
\begin{equation*}
c\geq J_{\lambda }(u_{n})
\end{equation*}%
\begin{equation*}
\geq \lbrack \frac{N-2}{2N}-\lambda \frac{N-q}{Nq}\Lambda ^{-\frac{q}{2}%
}V(M)^{1-\frac{q}{N}}(\max ((1+\varepsilon )K_{\circ },A_{\varepsilon }))^{%
\frac{q}{2}}\left\Vert u_{n}\right\Vert ^{q-2}]\left\Vert u_{n}\right\Vert
^{2}>0
\end{equation*}%
hence%
\begin{equation*}
0\leq \left\Vert u_{n}\right\Vert ^{2}\leq \frac{c}{\frac{N-2}{2N}-\lambda 
\frac{N-q}{Nq}\Lambda ^{-\frac{q}{2}}V(M)^{1-\frac{q}{N}}(\max
((1+\varepsilon )K_{\circ },A_{\varepsilon }))^{\frac{q}{2}}\rho ^{q-2}}%
<+\infty \text{.}
\end{equation*}%
So $\left( u_{n}\right) _{n}$ is bounded in $H_{2}^{2}(M)$.\ Since \ $%
H_{2}^{2}(M)$ is reflexive and the embedding $H_{2}^{2}(M)\subset \
H_{p}^{k}(M)$ \ ($k\ =0,1$; $p<N$ ) is compact and we have

. $u_{n}\rightarrow u$\ \ weakly in $H_{2}^{2}(M)$.

. $u_{n}\rightarrow u$ strongly in $\ H_{p}^{k}(M)$\ ; $p<N$.

. $u_{n}\rightarrow u$ a.e. in $M$.\newline
The Brezis-Lieb lemma allows us to write 
\begin{equation*}
\int_{M}\left\vert \Delta u_{n}\right\vert ^{2}dv_{g}=\int_{M}\left\vert
\Delta u\right\vert ^{2}dv_{g}+\int_{M}\left\vert \Delta
(u_{n}-u)\right\vert ^{2}dvg+o(1)
\end{equation*}%
and also 
\begin{equation*}
\int_{M}f(x)\left\vert u_{n}\right\vert ^{N}dv_{g}=\int_{M}f(x)\left\vert
u\right\vert ^{N}dv_{g}+\int_{M}f(x)\left\vert u_{n}-u\right\vert
^{N}dv_{g}+o(1)\text{.}
\end{equation*}%
We claim that $u\in M_{\lambda }$, indeed since

$u_{n}\rightarrow u$\ \ weakly in $H_{2}^{2}(M)$, we have for any $\phi \in
H_{2}^{2}(M)$,%
\begin{equation*}
\int_{M}\left( \Delta u_{n}\Delta \phi -a(x)\left\langle \nabla u_{n},\nabla
\phi \right\rangle +au_{n}\phi \right) \text{ }dv_{g}=
\end{equation*}

\begin{equation*}
\int_{M}\left( \Delta u\Delta \phi -a(x)\left\langle \nabla u,\nabla \phi
\right\rangle +au\phi \right) dv_{g}+o(1)
\end{equation*}%
and in particular if we let $\phi =u$, 
\begin{equation*}
\int_{M}\left( \Delta u_{n}\Delta u-a(x)\left\langle \nabla u_{n},\nabla
u\right\rangle +au_{n}u\right) dv_{g}=\left\Vert u\right\Vert ^{2}+o(1)
\end{equation*}%
and also if we put $\phi =u_{n}$, we obtain%
\begin{equation*}
\int_{M}\left( \Delta u_{n}\Delta u-a(x)\left\langle \nabla u_{n},\nabla
u\right\rangle +au_{n}u\right) dv_{g}=\left\Vert u_{n}\right\Vert ^{2}+o(1)%
\text{.}
\end{equation*}%
Now since $(u_{n})_{n}$ belongs to $M_{\lambda }$, we get%
\begin{equation*}
\int_{M}\left( \lambda \left\vert u_{n}\right\vert
^{q-2}u_{n}u+f(x)\left\vert u_{n}\right\vert ^{N-2}u_{n}u\right) dv(g)\
=\left\Vert u\right\Vert ^{2}+o(1)
\end{equation*}%
and by letting $n\rightarrow +\infty $%
\begin{equation*}
\int_{M}\left( \lambda \left\vert u_{n}\right\vert
^{q-2}u_{n}u+f(x)\left\vert u_{n}\right\vert ^{N-2}u_{n}u\right)
dv(g)\rightarrow \int_{M}\left( \lambda \left\vert u\right\vert
^{q}+f(x)\left\vert u\right\vert ^{N}\right) dv(g)\text{.}
\end{equation*}%
Hence%
\begin{equation*}
\Phi _{\lambda }(u_{n})=\Phi _{\lambda }(u)=\left\Vert u\right\Vert
^{2}-\lambda \int_{M}\left\vert u\right\vert
^{q}dv(g)-\int_{M}f(x)\left\vert u\right\vert ^{N}dv(g)=0
\end{equation*}%
and we have 
\begin{equation*}
\left\Vert u\right\Vert +o(1)=\left\Vert u_{n}\right\Vert \geq \rho \text{.}
\end{equation*}%
Consequently $u\in M_{\lambda }$.

Also we claim that $\mu _{n}\rightarrow 0$ as $n$ $\rightarrow +\infty $ in
fact testing with $u_{n}$, we get 
\begin{equation*}
\left\langle \nabla J_{\lambda }(u_{n})-\mu _{n}\nabla \Phi _{\lambda
}(u_{n}),u_{n}\right\rangle =o(1)
\end{equation*}

\begin{equation*}
=\underset{=0}{\underbrace{\left\langle \nabla J_{\lambda
}(u_{n}),u_{n}\right\rangle }}-\mu _{n}\left\langle \nabla \Phi _{\lambda
}(u_{n}),u_{n}\right\rangle =o(1)\text{.}
\end{equation*}%
hence 
\begin{equation*}
\mu _{n}\left\langle \nabla \Phi _{\lambda }(u_{n}),u_{n}\right\rangle =o(1)
\end{equation*}%
and by Lemma \ref{lem3}, we have%
\begin{equation*}
\lim \sup_{n}\left\langle \nabla \Phi _{\lambda }(u_{n}),u_{n}\right\rangle
<0\newline
\end{equation*}%
so $\mu _{n}\rightarrow 0$ as $n\rightarrow +\infty $.\newline
We are going to show now that $u_{n}\rightarrow u$ converges strongly in $%
H_{2}^{2}(M)$. First we have%
\begin{equation*}
J_{\lambda }(u_{n})-J_{\lambda }(u)
\end{equation*}%
\begin{equation}
=\frac{1}{2}\int_{M}\left( \Delta (u_{n}-u)\right) ^{2}dv_{g}-\frac{1}{N}%
\int_{M}f(x)\left\vert u_{n}-u\right\vert ^{N}dv_{g}+o(1)  \tag{10}
\label{10}
\end{equation}%
and since $u_{n}-u\rightarrow 0$\ converges weakly in $H_{2}^{2}(M)$, by
testing $\nabla J_{\lambda }(u_{n})-\nabla J_{\lambda }(u)$, we get

\begin{equation*}
\left\langle \nabla J_{\lambda }(u_{n})-\nabla J_{\lambda
}(u),u_{n}-u\right\rangle =o(1)
\end{equation*}

\begin{equation*}
=\int_{M}\left( \Delta (u_{n}-u)\right) ^{2}dv_{g}-\int_{M}f(x)\left\vert
u_{n}-u\right\vert ^{N}dv_{g}=o(1)
\end{equation*}%
that is to say 
\begin{equation}
\int_{M}\left( \Delta _{g}(u_{n}-u)\right) ^{2}dv_{g}=\int_{M}f(x)\left\vert
u_{n}-u\right\vert ^{N}dv_{g}+o(1)  \tag{11}  \label{11}
\end{equation}%
hence taking account of (\ref{10}), we obtain$\ $

\begin{equation*}
J_{\lambda }(u_{n})-J_{\lambda }(u)=\frac{2}{n}\int_{M}\left( \Delta
(u_{n}-u)\right) ^{2}dv_{g}+o\left( 1\right) .
\end{equation*}%
The Sobolev inequality allows us to write

\begin{equation*}
\left\Vert u_{n}-u\right\Vert _{N}^{2}\leq (1+\varepsilon )K_{\circ
}\int_{M}\left( \Delta (u_{n}-u)\right) ^{2}dv_{g}+o(1)
\end{equation*}%
so%
\begin{equation}
\int_{M}f(x)\left\vert u_{n}-u\right\vert ^{N}dv_{g}\leq (1+\varepsilon )^{%
\frac{n}{n-4}}\max_{x\in M}f(x)K_{\circ }^{\frac{n}{n-4}}\left\Vert \Delta
(u_{n}-u)\right\Vert _{2}^{N}+o(1)\text{.}  \tag{12}  \label{12}
\end{equation}%
Taking account of equality (\ref{11}), one writes

\begin{equation*}
o(1)\geq \left\Vert \Delta (u_{n}-u)\right\Vert _{2}^{2}-(1+\varepsilon )^{%
\frac{n}{n-4}}\max_{x\in M}f(x)K_{\circ }^{\frac{n}{n-4}}\left\Vert \Delta
(u_{n}-u)\right\Vert _{2}^{N}+o(1)
\end{equation*}

\begin{equation*}
\geq \left\Vert \Delta (u_{n}-u)\right\Vert _{2}^{2}(\left( 1-(1+\varepsilon
)^{\frac{n}{n-4}}\max_{x\in M}f(x)K_{\circ }^{\frac{n}{n-4}}\left\Vert
\Delta (u_{n}-u)\right\Vert _{2}^{N-2}\right) +o(1)\text{.}
\end{equation*}%
Hence if 
\begin{equation*}
\lim \sup_{n}\left\Vert \Delta (u_{n}-u)\right\Vert _{2}^{N-2}<\frac{1}{%
\left( (1+\varepsilon )K_{o}\right) ^{\frac{n}{n-4}}\max_{x\in M}f(x)}
\end{equation*}%
we get 
\begin{equation*}
\frac{2}{n}\int_{M}\left\vert \Delta _{g}(u_{n}-u)\right\vert ^{2}dv(g)<c%
\text{.}
\end{equation*}%
Since 
\begin{equation*}
c<\ \frac{2}{n\text{ }K_{\circ }^{\frac{n}{4}}(\max \text{ }_{x\in M}f(x))^{%
\frac{n-4}{4}}}
\end{equation*}%
it follows that 
\begin{equation*}
\int_{M}\left( \Delta (u_{n}-u)\right) ^{2}dv_{g}<\frac{1}{K_{\circ }^{\frac{%
n}{4}}(\max_{x\in M}f(x))^{\frac{n-4}{4}}}\text{.}
\end{equation*}%
Consequently 
\begin{equation*}
o(1)\geq \left\Vert \Delta (u_{n}-u)\right\Vert _{2}^{2}\underbrace{%
(1-(1+\varepsilon )^{\frac{n}{n-4}}\max_{x\in M}f(x)K_{\circ }^{\frac{n}{n-4}%
}\left\Vert \Delta (u_{n}-u)\right\Vert _{2}^{N-2})}_{>0\newline
}+o(1)
\end{equation*}%
\ \ \ \ \ \ \ \ \ \ \ \ \ \ \ \ \ \ \ \ \ \ \ \ \ \ \ \ \ \ \ \ \ \ \ \ \ \
\ \ \ \ \ \ \ \ \ \ \ \ \ \ \ \ \ \ \ \ \ \ \ which shows that $\ $%
\begin{equation*}
\left\Vert \Delta (u_{n}-u)\right\Vert _{2}^{2}=o(1)
\end{equation*}%
and $u_{n}\rightarrow u$ converges strongly in $H_{2}^{2}(M)$.
\end{proof}

\begin{theorem}
\label{th3} Let $\left( M,g\right) $ be an $n$-dimensional compact
Riemannian manifold with $n\geqslant 6$ and $f$ be a smooth positive
function. Assume that the operator $P(u)=\Delta ^{2}u+\bigtriangledown
^{i}\left( a(x)\bigtriangledown _{i}u\right) +b(x)u$ is coercive and 
\begin{equation*}
c<\frac{2}{nK_{o}^{\frac{n}{4}}\max_{x\in M}f(x)^{\frac{n}{4}-1}}
\end{equation*}

Then there exists a $\Lambda >0$ such that for any $\lambda \in \left(
0,\Lambda \right) $, equation (\ref{1}) have a non trivial weak solution.
\end{theorem}

\begin{proof}
According to Lemma \ref{lem2}, Lemma \ref{lem3} and Lemma \ref{lem4}, we
infer the existence of $v\in M_{\lambda }$ such that $J_{\lambda
}(v)=\min_{u\in M_{\lambda }}J_{\lambda }$. So there is a real $\mu $ such
that 
\begin{equation*}
\bigtriangledown J_{\lambda }(v)=\mu \bigtriangledown Q_{\lambda }(v)
\end{equation*}%
and multiplying by $v$ and taking account of Lemma\ref{lem3} we obtain that $%
\mu =0$. Hence $v$ is a non solution of equation (\ref{1}).
\end{proof}

\section{Multiplicity of solutions in case of constant coefficients}

When $P_{g}$ has constant coefficients, we set

\begin{equation*}
J_{\lambda }^{+}(u)=\frac{1}{2}\left( \left\Vert \Delta u\right\Vert
_{2}^{2}-\alpha \int_{M}\left\vert \nabla u\right\vert ^{2}dv_{g}+\beta
\int_{M}u^{2}dv_{g}\right) -\frac{\lambda }{q}\int_{M}u^{+q}dv_{g}\text{\ \
\ }
\end{equation*}%
\begin{equation*}
-\frac{1}{N}\int_{M}f(x)u^{+N}dv_{g}
\end{equation*}%
where 
\begin{equation*}
u^{+}=\max \left( u,0\right) \text{ }
\end{equation*}%
Critical points of $J_{\lambda }^{+}$ are solutions to%
\begin{equation}
\Delta ^{2}u+\alpha \Delta u+\beta u=\lambda \left( \left( u^{+}\right)
^{q-2}+f\left( u^{+}\right) ^{N-2}\right) u^{+}\text{.}  \tag{13}  \label{13}
\end{equation}%
Similar arguments as the ones used in the precedent sections give that $%
J_{\lambda }^{+}$ has a critical point $u$. Standard arguments show that $u$
is of class $C^{4,\theta }$ with $\theta \in \left( 0,1\right) $. If $\alpha
^{2}-4\beta >0$, we let $x_{1}=\frac{\alpha -\sqrt{\alpha ^{2}-4\beta }}{2}$
and $x_{2}=\frac{\alpha +\sqrt{\alpha ^{2}-4\beta }}{2}$ \ and moreover if $%
\alpha >0$, then $x_{1}$, $x_{2}>0$ and

\begin{equation*}
\left( \Delta +x_{1}\right) \left( \Delta +x_{2}\right) u=\Delta
^{2}u+\alpha \Delta u+\beta u\geq 0\text{.}
\end{equation*}%
Applying the maximum principle twice, we obtain that $u$ is a positive
solution of class $C^{4,\theta }$, where $\theta \in (0,1)$ of the equation%
\begin{equation*}
\Delta ^{2}u+\alpha \Delta u+\beta u=\lambda \left( u^{q-1}+fu^{N-1}\right) 
\text{.}
\end{equation*}%
and standard regularity results give that $u$ is smooth.

In the same manner if we set

\begin{equation*}
J_{\lambda }^{-}(u)=\frac{1}{2}\left( \left\Vert \Delta u\right\Vert
_{2}^{2}-\alpha \int_{M}\left\vert \nabla u\right\vert ^{2}dv_{g}+\beta
\int_{M}u^{2}dv_{g}\right) -\frac{\lambda }{q}\int_{M}\left\vert
u^{-q}\right\vert dv_{g}\text{\ \ \ }
\end{equation*}%
\begin{equation*}
-\frac{1}{N}\int_{M}f(x)\left\vert u^{-}\right\vert ^{N}dv_{g}
\end{equation*}%
where 
\begin{equation*}
u^{-}=\min \left( u,0\right) \text{ }
\end{equation*}%
then the critical points of $J_{\lambda }^{-}$ are solutions to%
\begin{equation*}
\Delta ^{2}u+\alpha \Delta u+\beta u=\lambda \left( \left\vert
u^{-}\right\vert ^{q-2}+f\left\vert u^{-}\right\vert ^{N-2}\right) u^{-}%
\text{.}
\end{equation*}%
By the \ same argument as above we get that $u^{-}$ is a negative smooth
solution. Similar arguments as the ones we used for $J_{\lambda }$ give that 
$J_{\lambda }^{+}$ and $J_{\lambda }^{-}$ have critical points $M_{\lambda
}^{+}$ and $M_{\lambda }^{-}$ respectively where 
\begin{equation*}
M_{\lambda }^{\pm }=\left\{ u\in H_{2}:Q_{\lambda }^{\pm }\left( u\right) =0%
\text{ and }\left\Vert u\right\Vert \geq \rho >0\right\}
\end{equation*}%
and%
\begin{equation*}
Q_{\lambda }^{\pm }\left( u\right) =\left\langle \bigtriangledown J_{\lambda
}^{\pm }(u),u\right\rangle \text{.}
\end{equation*}

Summarizing, we get

\begin{theorem}
\label{th4} Let $\left( M,g\right) $ be an $n$-dimensional compact
Riemannian manifold with $n\geqslant 6$. Assume that the operator $%
P(u)=\Delta ^{2}u+\bigtriangledown ^{i}\left( a(x)\bigtriangledown
_{i}u\right) +b(x)u$ is coercive and 
\begin{equation*}
c<\frac{2}{nK_{o}^{\frac{n}{4}}\max_{x\in M}f(x)^{\frac{n}{4}-1}}\text{.}
\end{equation*}

If moreover $\frac{\alpha ^{2}}{4}>\beta $ with $\alpha >0$. Then equation (%
\ref{5}) has two distinct smooth solutions; one positive and the other
negative.
\end{theorem}

\begin{lemma}
\label{lem6} For any $\lambda >0$, sufficiently small, $J_{\lambda }$ has
two local minima.
\end{lemma}

\begin{proof}
We follow closely the proof of Lemma 8 in \cite{2}. As a consequence of
Lemma \ref{lem2}, Lemma \ref{lem3} and Lemma \ref{lem5}, we infer the
existence of $v_{1}\in M_{\lambda }^{+}$ and a $v_{2}\in M_{\lambda }^{-}$
such that 
\begin{equation*}
J_{\lambda }^{+}(v_{1})=\min_{u\in M_{\lambda }^{+}}J_{\lambda }^{+}(u)
\end{equation*}%
and%
\begin{equation*}
J_{\lambda }^{-}(v_{2})=\min_{u\in M_{\lambda }^{-}}J_{\lambda }^{-}(u)\text{%
.}
\end{equation*}%
Note that $v_{1}$ and $v_{2}$ are respectively smooth positive and negative
solutions of equation (\ref{5}). Indeed by Lagrange mutiplicators theorem we
get that 
\begin{equation*}
\bigtriangledown J_{\lambda }^{+}\left( v_{1}\right) =\mu \bigtriangledown
Q_{\lambda }^{+}(v_{1})
\end{equation*}%
and multiplying by $v_{1}$ we deduce that%
\begin{equation*}
\mu =0
\end{equation*}%
Hence $v_{1}$ is a solution of (\ref{13}) and as in section 3 we get that $%
v_{1}$ positive, hence a positive solution of equation ( \ref{5}). $v_{2}$
is actually a negative solution of (\ref{5}). We claim that $v_{1}$ and $%
v_{2}$ are local minima of $J_{\lambda }$ if it is not the case let $%
w_{n}\in M_{\lambda }$ such that $w_{n}\rightarrow v_{1}$ in $H_{2}^{2}$ as $%
n\rightarrow +\infty $ and 
\begin{equation}
J_{\lambda }(w_{n})<J_{\lambda }^{+}\left( v_{1}\right)  \tag{14}  \label{14}
\end{equation}%
We can choose $w_{n}$ as 
\begin{equation}
J_{\lambda }(w_{n})=\inf_{u\in B_{n}\cap M_{\lambda }}J_{\lambda }(u) 
\tag{15}  \label{15}
\end{equation}%
where $B_{n}=\left\{ u\in \left\Vert u-v_{1}\right\Vert _{H_{1}^{2}}\leq 
\frac{1}{n}\right\} $. There exist parameters $\lambda _{n}$ and $\mu _{n}$
such that%
\begin{equation}
\bigtriangledown J_{\lambda }(w_{n})=\lambda _{n}\bigtriangledown Q_{\lambda
}(w_{n})+\mu _{n}\left( \Delta ^{2}w_{n}+\alpha \Delta w_{n}+\beta
w_{n}\right)  \tag{16}  \label{16}
\end{equation}%
with $\mu _{n}\leq 0$. Taking the inner product of the latter equality with $%
w_{n}$, we get%
\begin{equation*}
\lambda _{n}\left\langle \bigtriangledown Q_{\lambda
}(w_{n}),w_{n}\right\rangle +\mu _{n}\left\Vert w_{n}\right\Vert
_{H_{1}^{2}}^{2}=0
\end{equation*}%
and we infer that $\lambda _{n}\leq 0$.

Equation reads as 
\begin{equation*}
\Delta ^{2}w_{n}+\alpha \Delta w_{n}+\beta w_{n}=\frac{\left( -\lambda
_{n}-\mu _{n}\right) }{\left( 1-\lambda _{n}-\mu _{n}\right) }f\left\vert
w_{n}\right\vert ^{N-2}w_{n}\text{.}
\end{equation*}%
By standard methods, $w_{n}$ is of class $C^{4,\theta }$ , $0<\theta <1$.
Hence $w_{n}$ goes to $v_{1}$ in the $C^{2}$ topology, then $w_{n}>0$. So (%
\ref{15}) is a contradiction with (\ref{14}). Hence $v_{1}$ and $v_{2}$ are
respectively positive and negative solution of equation (\ref{5}) of minimal
positive energy.
\end{proof}

Next we prove

\begin{theorem}
\label{th5} Let $\left( M,g\right) $ be an $n$-dimensional compact
Riemannian manifold with $n\geqslant 6$. Assume that the operator $%
P(u)=\Delta ^{2}u+\bigtriangledown ^{i}\left( a(x)\bigtriangledown
_{i}u\right) +b(x)u$ is coercive and 
\begin{equation*}
c<\frac{2}{nK_{o}^{\frac{n}{4}}\max_{x\in M}f(x)^{\frac{n}{4}-1}}\text{.}
\end{equation*}

If moreover $\frac{\alpha ^{2}}{4}>\beta $ with $\alpha >0$. Then equation (%
\ref{5}) has third solution $w$ distinct of $u^{+}$ and $u^{-}$.
\end{theorem}

\begin{proof}
We can suppose that the minima of $J_{\lambda }$ are realized by $u^{+}$ and 
$u^{-}$ . The geometric conditions of the Mountain pass theorem are
satisfied. If $\Gamma $ denotes the set of paths $\gamma :\left[ 0,1\right]
\rightarrow M_{\lambda }$ such that $\gamma (0)=u^{-}$ and $\gamma (1)=u^{+}$%
. Let $c_{\lambda }=\inf_{\gamma \in \Gamma }\max_{t\in \left[ 0,1\right]
}\left( J_{\lambda }\left( \gamma )\right) \right) $. By Lemma \pageref{4},
we infer that $c_{\lambda }$ is a critical level of the function $J_{\lambda
}$ with critical value $w$ and by Lemma \ref{3} $w\in M_{\lambda }$. Hence $%
w $ is solution of equation (\ref{5}) different from $u^{+}$ and $u^{-}$.
\end{proof}

\section{Test functions}

In this section we give the proof of Theorem \ref{th1} and \ref{th2}.$%
\newline
$Let $(y^{1},...,y^{n})$ be normal coordinates centred at the point $%
x_{\circ }$ where the function attains its maximum and $S(r)$ be the
geodesic sphere centred at $x_{o}$ and of radius $r$ ( $r<d$ the injectivity
radius ). Denote by $d\sigma $ the volume element of the $n-1$-dimensional
unit $S^{n-1}$.

Put 
\begin{equation*}
G(r)=\frac{1}{w_{n-1}}\dint\limits_{S(r)}\sqrt{\left\vert g(x)\right\vert }%
d\sigma
\end{equation*}%
where $w_{n-1}$ denotes the area of $S^{n-1}$ and $\left\vert
g(x)\right\vert $ the determinant of the metric $g$. An expansion of $G(r)$
in a neighborhood of $r=0$ writes as 
\begin{equation*}
G(r)=1-\frac{S_{g}(x_{\circ })}{6n}r^{2}+o(r^{2})
\end{equation*}%
where $S_{g}(x_{\circ })$ denotes the scalar curvature of $M$ at the point $%
x_{\circ }$.$\newline
$Let $B(x_{\circ },\delta )$ be the ball centred at $x_{\circ }$ and of
radius $\delta $ with $0<2\delta <d$ and let $\eta $ be a smooth function
equals to $1$ on $B\left( x_{o},\delta \right) $ and equals to $0$ on $%
M-B\left( x_{o},2\delta \right) $.$\newline
$ Put

\begin{equation*}
u_{\epsilon }(x)=(\frac{(n-4)n(n^{2}-4)\epsilon ^{4}}{f(x_{\circ })})^{\frac{%
n-4}{8}}\frac{\eta (r)}{(r^{2}+\epsilon ^{2})^{\frac{n-4}{2}}}
\end{equation*}%
where \ 
\begin{equation*}
f(x_{\circ })=\max_{x\in M}f(x)
\end{equation*}%
and $r=d(x_{\circ },.)$ is geodesic distance to the point $x_{\circ }$.%
\newline
We let, for $p-q>1$,

\begin{equation*}
I_{p}^{q}=\int_{0}^{+\infty }\frac{t^{q}}{(1+t)^{p}}dt
\end{equation*}%
which fulfills 
\begin{equation*}
I_{p+1}^{q}=\frac{p-q-1}{p}I_{p}^{q}\text{ \ \ \ and \ \ \ \ }I_{p+1}^{q+1}=%
\frac{q+1}{p-q-1}I_{p+1}^{q}\text{.}
\end{equation*}%
In the case where the dimension of the manifold $n>6$, we have

\begin{theorem}
Let $\left( M,g\right) $ be an $n$-dimensional compact Riemannian manifold
with $n>6$. If at the point $x_{o}$ where the function $f$ achieves its
maximum 
\begin{equation*}
\frac{n\left( n^{2}+4n-20\right) }{2\left( n-2\right) \left( n^{2}-4\right) }%
S_{g}\left( x_{o}\right) +\frac{n\left( n-6\right) }{\left( n-2\right)
\left( n^{2}-4\right) }a\left( x_{o}\right) -\frac{n}{8\left( n-2\right) }%
\frac{\Delta f(x_{o})}{f\left( x_{o}\right) }>0
\end{equation*}
equation (\ref{1}) have a non trivial solution of class $C^{4,\alpha }(M)$, $%
\alpha \in (0,1)$.
\end{theorem}

\begin{proof}
As in (\cite{7}), we get 
\begin{equation*}
\int_{M}f(x)\left\vert u_{\epsilon }(x)\right\vert ^{N}dv(g)=\frac{1}{%
K_{\circ }^{\frac{n}{4}}(f(x_{\circ }))^{\frac{n-4}{4}}}\left( 1-(\frac{%
\Delta f(x_{\circ })}{2(n-2)f(x_{\circ })}+\frac{S_{g}(x_{\circ })}{6(n-2)}%
)\epsilon ^{2}+o(\epsilon ^{2})\right)
\end{equation*}%
and also%
\begin{equation*}
\int_{M}a(x)\left\vert \nabla u_{\epsilon }\right\vert ^{2}dv(g)=\frac{1}{%
K_{\circ }^{\frac{n}{4}}(f(x_{\circ }))^{\frac{n-4}{4}}}\left( \frac{%
4(n-1)a(x_{\circ })}{(n^{2}-4)(n-6)}\epsilon ^{2}+o(\epsilon ^{2})\right) 
\text{.}
\end{equation*}%
The computations give%
\begin{equation*}
\int_{M}b(x)u_{\epsilon }^{2}dv(g)=o(\epsilon ^{2})
\end{equation*}%
and%
\begin{equation*}
\int_{M}\left\vert \Delta u_{\epsilon }\right\vert ^{2}dv(g)=\frac{1}{%
K_{\circ }^{\frac{n}{4}}(f(x_{\circ }))^{\frac{n-4}{4}}}\left( 1-\frac{%
n^{2}+4n-20}{6(n^{2}-4)(n-6)}S_{g}(x_{\circ })\epsilon ^{2}+o(\epsilon
^{2})\right) \text{.}
\end{equation*}%
Summarizing we obtain%
\begin{equation*}
\int_{M}\left\vert \Delta u_{\epsilon }\right\vert ^{2}-a(x)\left\vert
\nabla u_{\epsilon }\right\vert ^{2}+b(x)u_{\epsilon }^{2}dv(g)=\frac{1}{%
K_{\circ }^{\frac{n}{4}}(f(x_{\circ }))^{\frac{n-4}{4}}}
\end{equation*}%
\begin{equation*}
\left( 1-(\frac{n^{2}+4n-20}{6(n^{2}-4)(n-6)}S_{g}(x_{\circ })+\frac{4(n-1)}{%
(n^{2}-4)(n-6)}a(x_{\circ }))\epsilon ^{2}+o(\epsilon ^{2})\right) \text{.}
\end{equation*}%
Taking in mind that 
\begin{equation*}
J_{\lambda }\left( u_{\epsilon }\right) =\frac{1}{2}\left\Vert u_{\epsilon
}\right\Vert ^{2}-\frac{\lambda }{q}\left\Vert u_{\epsilon }\right\Vert
_{q}^{q}-\frac{1}{N}\int_{M}f(x)\left\vert u_{\epsilon }(x)\right\vert
^{N}dv(g)
\end{equation*}%
where 
\begin{equation*}
\left\Vert u_{\epsilon }\right\Vert ^{2}=\int_{M}\left\vert \Delta
u_{\epsilon }\right\vert ^{2}-a(x)\left\vert \nabla u_{\epsilon }\right\vert
^{2}+b(x)u_{\epsilon }^{2}dv(g)
\end{equation*}%
and since $\lambda >0$, we get

\begin{equation*}
J_{\lambda }\left( u_{\epsilon }\right) \leq \frac{1}{2}\left\Vert
u_{\epsilon }\right\Vert ^{2}-\frac{1}{N}\int_{M}f(x)\left\vert u_{\epsilon
}(x)\right\vert ^{N}dv(g)
\end{equation*}%
\begin{equation*}
\leq \frac{1}{K_{\circ }^{\frac{n}{4}}(f(x_{\circ }))^{\frac{n-4}{4}}}\left[ 
\frac{2}{n}-\left( \frac{n^{2}+4n-20}{(n^{2}-4)(n-6)}S_{g}(x_{\circ })+\frac{%
2(n-1)}{(n^{2}-4)(n-6)}a(x_{\circ })-\frac{1}{4(n-2)}\frac{\Delta f(x_{\circ
})}{f(x_{\circ })}\right) \epsilon ^{2}+o(\epsilon ^{2})\right]
\end{equation*}%
\begin{equation*}
\leq \frac{2}{n\text{ }K_{\circ }^{\frac{n}{4}}(f(x_{\circ }))^{\frac{n-4}{4}%
}}\left[ 1-\left( \frac{\left( n^{2}+4n-20\right) n}{2(n^{2}-4)(n-6)}%
S_{g}(x_{\circ })+\frac{(n-1)\text{ }n}{(n^{2}-4)(n-6)}a(x_{\circ })-\frac{n%
}{8(n-2)}\frac{\Delta f(x_{\circ })}{f(x_{\circ })}\right) \epsilon
^{2}+o(\epsilon ^{2})\right] \text{.}
\end{equation*}%
So the condition 
\begin{equation*}
J_{\lambda }\left( u_{\epsilon }\right) <\ \frac{2}{n\text{ }K_{\circ }^{%
\frac{n}{4}}(f(x_{\circ }))^{\frac{n-4}{4}}}
\end{equation*}%
is fulfilled if 
\begin{equation*}
\left( \frac{\left( n^{2}+4n-20\right) n}{2(n^{2}-4)(n-6)}S_{g}(x_{\circ })+%
\frac{(n-1)\text{ }n}{(n^{2}-4)(n-6)}a(x_{\circ })-\frac{n}{8(n-2)}\frac{%
\Delta f(x_{\circ })}{f(x_{\circ })}\right) >0\text{.}
\end{equation*}
\end{proof}

In the case $n=6$, we have

\begin{proof}
The same calculations as in case $n>6$ lead to \ 
\begin{equation}
\int_{M}f(x)\left\vert u_{\epsilon }(x)\right\vert ^{N}dv(g)=\frac{1}{%
K_{\circ }^{\frac{n}{4}}(f(x_{\circ }))^{\frac{n-4}{4}}}\left( 1-(\frac{%
\Delta f(x_{\circ })}{2(n-2)f(x_{\circ })}+\frac{S_{g}(x_{\circ })}{6(n-2)}%
)\epsilon ^{2}+o(\epsilon ^{2})\right) \text{.}  \tag{17}  \label{17}
\end{equation}%
Also the same computations as in \cite{7} with minor modifications allow us
to write%
\begin{equation*}
\int_{M}a(x)\left\vert \nabla u_{\epsilon }\right\vert ^{2}dv(g)=(n-4)^{2}(%
\frac{(n-4)n(n^{2}-4)}{f(x_{\circ })})^{\frac{n-4}{4}}\frac{w_{n-1}}{2}%
\left( a(x_{\circ })\epsilon ^{2}\log (\frac{1}{\epsilon ^{2}})\text{ }%
+O(\epsilon ^{2})\right)
\end{equation*}%
and 
\begin{equation*}
\int_{M}\left\vert \Delta u_{\epsilon }\right\vert ^{2}dv(g)=(n-4)^{2}(\frac{%
(n-4)n(n^{2}-4)}{f(x_{\circ })})^{\frac{n-4}{4}}\frac{w_{n-1}}{2}
\end{equation*}%
\begin{equation*}
\left( \frac{n(n+2)(n-2)}{(n-4)}I_{n}^{\frac{n}{2}-1}-\frac{2}{n}%
S_{g}(x_{\circ })\epsilon ^{2}\log (\frac{1}{\epsilon ^{2}})+O(\epsilon
^{2})\right) \text{.}
\end{equation*}%
Consequently%
\begin{equation*}
\int_{M}\left( \Delta u_{\epsilon }\right) ^{2}-a(x)\left\vert \nabla
u_{\epsilon }\right\vert ^{2}+b(x)u_{\epsilon }^{2}dv(g)=(n-4)^{2}(\frac{%
(n-4)n(n^{2}-4)}{f(x_{\circ })})^{\frac{n-4}{4}}\frac{w_{n-1}}{2}
\end{equation*}%
\begin{equation*}
\left[ \frac{n(n+2)(n-2)}{(n-4)}I_{n}^{\frac{n}{2}-1}-\left( \frac{2}{n}%
S_{g}(x_{\circ })+a(x_{\circ })\right) \epsilon ^{2}\log (\frac{1}{\epsilon
^{2}})+O(\epsilon ^{2})\right]
\end{equation*}%
\begin{equation*}
=\frac{1}{K_{\circ }^{\frac{n}{4}}(f(x_{\circ }))^{\frac{n-4}{4}}}\left( 1-%
\frac{\left( n-4\right) }{n\left( n^{2}-4\right) I_{n}^{\frac{n}{2}-1}}%
\left( \frac{2}{n}S_{g}(x_{\circ })+a(x_{\circ })\right) \epsilon ^{2}\log (%
\frac{1}{\epsilon ^{2}})+O(\epsilon ^{2})\text{ \ }\right)
\end{equation*}%
and taking account of (\ref{17}), we obtain%
\begin{equation*}
J_{\lambda }\left( u_{\epsilon }\right) \leq \frac{1}{2}\left\Vert
u_{\epsilon }\right\Vert ^{2}-\frac{1}{N}\int_{M}f(x)\left\vert u_{\epsilon
}(x)\right\vert ^{N}dv(g)
\end{equation*}

\begin{equation*}
\leq \frac{1}{K_{\circ }^{\frac{n}{4}}(f(x_{\circ }))^{\frac{n-4}{4}}}\left( 
\frac{1}{2}-\frac{1}{N}-\frac{\left( n-4\right) t_{o}^{2}}{2n\left(
n^{2}-4\right) I_{n}^{\frac{n}{2}-1}}\left( \frac{2}{n}S_{g}(x_{\circ
})+a(x_{\circ })\right) \epsilon ^{2}\log (\frac{1}{\epsilon ^{2}}%
)+O(\epsilon ^{2})\right) \text{.}
\end{equation*}%
So if in the point $x_{o}$ where the maximum of the function $f$ is
achieved, the condition $\frac{2}{n}S_{g}(x_{\circ })+a(x_{\circ })>0$ i.e.
since $n=6$, $S_{g}(x_{\circ })>-3a(x_{\circ })$ is fulfilled, we get for $%
\epsilon $ sufficiently small%
\begin{equation*}
J_{\lambda }\left( u_{\epsilon }\right) <\frac{2}{nK_{\circ }^{\frac{n}{4}%
}(f(x_{\circ }))^{\frac{n-4}{4}}}\text{.}
\end{equation*}
\end{proof}

\end{document}